\documentclass[a4paper, 10pt]{amsart}
\usepackage{mathtools, microtype, mathrsfs, xfrac, amssymb, enumerate, xspace}
\usepackage{graphicx}
\usepackage{float}
\usepackage[export]{adjustbox}
\usepackage{subcaption}
\usepackage{tikz-cd}
\usepackage[percent]{overpic}
\usepackage[all,cmtip]{xy}
\usepackage{hyperref}
\hypersetup{
	colorlinks=true,
	linkcolor=blue,
	filecolor=magenta,      
	urlcolor=cyan,
}
\usepackage{faktor}
\usepackage{cleveref}
\usepackage{comment}
\usepackage[
backend=biber,
style=alphabetic,
sorting=nyt
]{biblatex}
\numberwithin{equation}{section}

\numberwithin{subsection}{section}

\allowdisplaybreaks[1]

\newtheorem*{namedtheorem}{\theoremname}
\newcommand{\theoremname}{testing}

\newtheorem*{maintheorem}{Theorem}

\newtheorem{theorem}{Theorem}[section]
\makeatletter
\newcommand*{\dupcntr}[2]{%
  \expandafter\let\csname c@#1\expandafter\endcsname\csname c@#2\endcsname
}
\makeatother
\dupcntr{theorem}{subsection}
\newtheorem{proposition}[theorem]{Proposition}
\newtheorem{proposition-definition}[theorem]
{Proposition-Definition}
\newtheorem{corollary}[theorem]{Corollary}
\newtheorem{lemma}[theorem]{Lemma}

\theoremstyle{definition}

\newtheorem{remark}[theorem]{Remark}

\theoremstyle{remark}

\newcommand\calM{\mathcal{M}}
\renewcommand{\mathcal}{\mathscr}

\newcommand\cK{\mathcal{K}} \newcommand\cL{\mathcal{L}}
\newcommand\cM{\mathcal{M}} 
\newcommand\cO{\mathcal{O}} \newcommand\cP{\mathcal{P}}

\newcommand\CC{\mathbb{C}}

\newcommand\QQ{\mathbb{Q}}

 \newcommand\ZZ{\mathbb{Z}}

\newcommand\bA{\mathbf{A}} 
\newcommand\bC{\mathbf{C}}

 \newcommand\bP{\mathbf{P}}
\newcommand\bQ{\mathbf{Q}}

\newcommand\rmc{\mathrm{c}}

\newcommand\arr{\ifinner\to\else\longrightarrow\fi}

\newcommand\arrto{\ifinner\mapsto\else\longmapsto\fi}

\newcommand{\xarr}{\xrightarrow}

\newcommand{\eqdef}{\mathrel{\smash{\overset{\mathrm{\scriptscriptstyle def}} =}}}

\def\displaytimes_#1{\mathrel{\mathop{\times}\limits_{#1}}}

\def\displayotimes_#1{\mathrel{\mathop{\bigotimes}\limits_{#1}}}

\newcommand\ext{\operatorname{Ext}}

\newcommand\aut{\operatorname{Aut}}

\newcommand\pic{\operatorname{Pic}}

\newcommand\spec{\operatorname{Spec}}

\newcommand\id{\mathrm{id}}

\newcommand\pr{\operatorname{pr}}

\newlength{\ignora}

\renewcommand{\setminus}{\smallsetminus}

\newcommand{\PGL}{\mathrm{PGL}}

\DeclareFontFamily{U}{mathx}{\hyphenchar\font45}
\DeclareFontShape{U}{mathx}{m}{n}{
	<5> <6> <7> <8> <9> <10>
	<10.95> <12> <14.4> <17.28> <20.74> <24.88>
	mathx10
}{}
\DeclareSymbolFont{mathx}{U}{mathx}{m}{n}
\DeclareFontSubstitution{U}{mathx}{m}{n}
\DeclareMathAccent{\widecheck}{0}{mathx}{"71}
\DeclareMathAccent{\wideparen}{0}{mathx}{"75}

\renewcommand{\epsilon}{\varepsilon}

\newcommand{\sing}{{\rm sing}}

\newcommand{\ssi}{{\rm ss}}

\newcommand{\sm}{{\rm sm}}

\begin{document}
\title[Integral Picard group of moduli of polarized K3 surfaces]{Integral Picard group of moduli\\of polarized K3 surfaces}
\author[A. Di Lorenzo]{Andrea Di Lorenzo}
	\address[A. Di Lorenzo]{Humboldt Universit\"{a}t zu Berlin, Germany}
	\email{andrea.dilorenzo@hu-berlin.de}
	\author[R. Fringuelli]{Roberto Fringuelli}
	\address[R. Fringuelli]{Universit\`{a} di Roma ``La Sapienza", Italy}
	\email{r.fringuelli@uniroma1.it}
	\author[A. Vistoli]{Angelo Vistoli}
	\address[A. Vistoli]{Scuola Normale Superiore, Pisa, Italy}
	\email{angelo.vistoli@sns.it}
 	\thanks{The third author was partially supported by research funds from Scuola Normale Superiore, and by PRIN project ``Derived and underived algebraic stacks and applications''. We thank the anonymous referee for providing very useful comments and suggestions, and for catching a mistake in a previous version of the manuscript.} 
	\maketitle
	\begin{abstract}
		We compute the integral Picard group of the moduli stack of polarized K3 surfaces of fixed degree whose singularities are at most rational double points, and of its coarse moduli space. We also compute the integral Picard group of the stack of quasi-polarized K3 surfaces, and of the stacky period domain.
	\end{abstract}
\section*{Introduction}
A very interesting invariant of a moduli stack is its Picard group. It was introduced by Mumford
in \cite{Mum65}, where he also computed the Picard group of the moduli stack of elliptic curves. This calculation prompted a great amount of research in this topic, that eventually leaded to a complete understanding of the Picard group of the moduli stack of curves over fields of almost every characteristic (see \cite{Har83}, \cite{AC87}, \cite{Vis98}, \cite{diL}, \cite{FV-Mg}). In particular, knowing the Picard group of a Deligne-Mumford stack with finite inertia also gives a description of the rational Picard group of the coarse moduli space. Integral Picard group of other interesting moduli stacks have also been computed in recent years (\cite{AI},\cite{CL} and \cite{DL-CI}).

Another quite relevant moduli stack is the moduli stack of polarized K3 sufaces. In particular, the rational Picard group of the moduli space $M_d$ of (primitively) polarized K3 surfaces of degree $d$ with at most rational double points has been the subject of much research \cites{Bru,MP}, eventually culminated in the proof of the so called Noether-Lefschetz conjecture \cite{BLMM}, from which one can deduce the rank of $\pic(M_d)\otimes\mathbb{Q}$. On the other hand, not much is known on the integral Picard group of the associated moduli stack $\cM_d$. In this paper we prove the following (we work over $\mathbb{C})$.
\begin{maintheorem}[\Cref{thm:pic Md}]
    Let $\cM_d$ be the moduli stack of primitively polarized K3 surfaces of degree $d$ with at most rational double points. Then we have
    \[ \pic(\cM_d) \simeq \ZZ^{\rho(d)}, \]
    where $\rho(d)$ is the rank of $\pic(M_d)\otimes\QQ$ computed in \cite{Bru}.
\end{maintheorem}

Furthermore, we prove that the integral Picard group of the moduli space $M_d$ is torsion free. 
\begin{maintheorem}[\Cref{cor:pic Md}] Let $M_d$ be the moduli space of primitively polarized K3 surfaces of degree $d$ with at most rational double points. Then we have
$$
\pic(M_d)\cong \mathbb Z^{\rho(d)}.
$$    
\end{maintheorem}
There are other two stacks that are closely related to $\cM_d$, namely the stack $\cK_d$ of primitively quasi-polarized K3 surfaces, and the stacky period domain $\cP_d$. At the level of schemes the differences between these stacks do not appear (indeed $\cP_d$ and $\cM_d$ have the same coarse moduli space), but as stacks they are all non isomorphic. Therefore, it makes sense to also ask what their integral Picard groups are. We give an answer in the following.
\begin{maintheorem}[\Cref{thm:pic Pd abstract}, \Cref{thm:pic Pd}, \Cref{thm:pic Pd new}]
    The following hold true:
    \begin{enumerate}
    \item As an abstract group, $\pic(\cP_{d}) \simeq \ZZ^{\rho(d)} \oplus \ZZ/2$.
    \item The morphism $\cK_{d} \arr \cP_{d}$ induces an isomorphism $\pic(\cP_d)\simeq\pic(\cK_d)$.
        \item Suppose that $\frac{d}{2}\not\equiv 1\pmod 4$: then we have a split short exact sequence
    \[ 0 \longrightarrow \pic(\cM_d) \longrightarrow \pic(\cP_d) \longrightarrow \ZZ/2 \longrightarrow 0. \] 
    \item Suppose $\frac{d}{2}\equiv 1\pmod 4$: then we have a non-split short exact sequence
    \[ 0 \longrightarrow \pic(\cM_d)\times\ZZ/2 \longrightarrow \pic(\cP_d) \longrightarrow \ZZ/2 \longrightarrow 0.\]
    \end{enumerate}
    
\end{maintheorem}
The generator of the torsion part in the Picard groups above is made explicit in the paper.

Notice that our proof does not give any hint as to what the generators of $\pic(\cM_{d})$ are. For $d \leq  8$ this is worked out in \cite{DiL-K3}, but for higher values of $d$ the problem is wide open.

\subsection*{Structure of the paper}
The paper is organized as follows. In \Cref{sec:moduli}, after introducing the moduli stacks we are interested in and after discussing some of their properties, we first show that there exists a morphism $\cP_d\to\cM_d$ from the stacky period domain (\Cref{lm:factorization}) to the stack of polarized K3 surfaces with rational double points, and we show that it induces an injection of Picard groups.  

Then in \Cref{sec:computation 1} we compute the torsion part of $\pic(\cP_d)$ by looking at the fundamental group of this stack (\Cref{prop:2 torsion Pic Pd}), and then we prove that the torsion line bundle on $\cP_d$ does not come from $\cM_d$ (\Cref{lemma:does not descend}).

After proving that the Picard group of $\cM_d$ is finitely generated, we obtain the desired conclusion. We then leverage the result just obtained to compute in \Cref{sec:computation 2} the Picard groups of $\cK_d$ and $\cP_d$ by means of certain localization exact sequences (\Cref{thm:pic Pd}).

\subsection*{Assumptions}
In what follows, we always work over $\mathbb{C}$.
\section{Some moduli stacks of K3 surfaces}\label{sec:moduli}
\subsection{}
In this section we introduce three different stacks, all of which in a sense parametrize polarized K3 surfaces of a fixed degree.
\subsection{}
Let $\cK_d$ be the stack of primitively quasi-polarized K3 surfaces of degree $d$. That is, the objects of $\cK_d$ over a scheme $S$ are pairs $(X\to S, L)$, where:
\begin{itemize}
    \item $X\to S$ is a proper, finitely presented and flat morphism whose geometric fibers are smooth K3 surfaces;
    \item $L$ is a section of $\underline{\pic}_{X/S} \to S$ that on the geometric fibers is represented by a primitive, numerically effective line bundle of degree $d$; we also require that if $\langle L_s, C_s\rangle =0$ for a curve $C_s \subset X_s$, where $s$ is a geometric point of $S$, then $(C_s^2)=-2$.
\end{itemize}
The morphisms in $\cK_d$ are given by $S$-isomorphisms $f:X\overset{\simeq}{\to} X'$ such that $f^*L'=L$. The fibred category $\cK_d$ is a smooth Deligne-Mumford stack \cite[(1.2.1), (1.2.2)]{Ols} (note that in \emph{loc. cit.} the stack $\cK_d$ is denoted $\mathbb{M}^{\sm}_d$). From now on, we will refer to the objects of $\cK_d$ as quasi-polarized K3 surfaces instead of primitively quasi-polarized K3 surfaces.
\subsection{}
Let $\cM_d$ denote the stack of primitively polarized K3 surfaces of degree $d$ with at most rational double points. That is, the objects of $\cM_d$ over a scheme $S$ are pairs $(X\to S, L)$ where:
\begin{itemize}
    \item $X\to S$ is a proper, finitely presented and flat morphism whose geometric fibers are K3 surfaces with at most rational double points;
    \item $L$ is a section of $\underline{\pic}_{X/S} \to S$ that on the geometric fibers is represented by an ample, primitive line bundle of degree $d$.
\end{itemize}
The morphisms in $\cM_d$ are given by $S$-isomorphisms $f:X\overset{\simeq}{\to} X'$ such that $f^*L'=L$. The fibred category $\cM_d$ is a smooth Deligne-Mumford stack with a coarse moduli space, which we denote $M_d$ \cite[84]{Huy} (note that in \emph{loc. cit.} the stack $\cM_d$ is the one denoted $\bar{\calM}_d$).
\subsection{}
Let $\Lambda$ denote the lattice $E_8(-1)^{\oplus 2}\oplus U^{\oplus 3}$. Given a smooth K3 surface $X$, this lattice arises as $H^2(X,\ZZ)$ together with the cohomological pairing.

Let $\ell$ be the element in $\Lambda$ defined as $e+\frac{d}{2}f$, where $e$ and $f$ form a basis for the first copy of $U$. Then we denote $\Lambda_d:=\ell^{\perp}$ the sublattice of $\Lambda$ orthogonal to $\ell$. Given a smooth K3 surface $X$ with a primitive quasi-polarization $L$ of degree $d$, then $\Lambda_d$ arises as the orthogonal of $c_1(L)$ in $H^2(X,\ZZ)$.
\subsection{}
Let $\Omega_d$ be the period domain of (primitively) quasi-polarized K3 surfaces of degree $d$, that is
\[\Omega_d := \left\{ \omega \in \bP(\Lambda_d\otimes\bC) \text{ such that } \langle \omega,\omega \rangle >0\text{, }\langle\omega,\overline{\omega}\rangle >0 \right\}.\]
This is a complex manifold, it has two connected components \cite[(1.2)]{Kon}. Let $D_d$ denote one connected component. Then $D_d$ is a bounded symmetric domain of type IV, hence simply connected \cite[(1.2)]{Kon}.
\subsection{}
Let $O(\Lambda_d)$ be the group of orthogonal transformations of $\Lambda_d$. Set 
\[ \widetilde{O}(\Lambda_d) = \ker (O(\Lambda_d) \longrightarrow O(\Lambda_d^{\vee}/\Lambda_d) ) \]
where $\Lambda_d^{\vee}$ is the lattice formed by $x\in \Lambda_d\otimes\bQ$ such that $\langle x,\Lambda \rangle \subset \ZZ$. Let $O^+(\Lambda_d)$ be the subgroup of orthogonal transformations having positive spinor norm. We set $\Gamma_d:=\widetilde{O}(\Lambda_d)\cap O^+(\Lambda_d)$; this arithmetic group can be regarded as the group of orthogonal transformations of $\Lambda$ that fix $\ell$, and it acts on the connected component $D_d$ with a properly discontinuous action \cite[(1.2)]{Kon}, \cite[16]{GHS}.
\subsection{}
We define the analytic quotient stack
\[ \cP_d := [ \Gamma_d \diagdown D_d], \]
and we refer to this stack as the \emph{period stack}. It is actually a smooth Deligne-Mumford stack \cite[Theorem 10.11]{BB}.
\subsection{}
There is a morphism of stacks
\[p_d: \cK_d \longrightarrow \cP_d \]
which is the stacky version of the usual period map. Indeed, given a quasi-polarized K3 surface $(\pi:X\to S,L)$, consider the associated analytic morphism $\pi^{an}:X^{an}\to S^{an}$ and the family of lattices $R^2\pi^{an}_* \ZZ$ together with the cohomological pairing. The quasi-polarization $L$ defines a section of $R^2\pi^{an}_* \ZZ$.

We can use this object to define a $\Gamma_d$-torsor $U_d\to\cK_d$: its objects are triples $(\pi:X\to S, L, \alpha)$, where $\alpha:(R^2\pi^{an}_*\ZZ,L) \simeq (\Lambda, \ell)$ is an isomorphism of lattices that sends $L$ to $\ell$ (the marking of the K3 surface). We can then construct a $\Gamma_d$-equivariant morphism $U_d\to D_d$ by sending a triple $(\pi:X\to S,L,\alpha)$ to the  line subbundle $\alpha(\pi_*\Omega) \subset \Lambda_d\otimes \cO_S$. The resulting morphism $p_d:\cK_d \to [\Gamma_d \diagdown D_d ]$ is \'etale and representable \cite[(1.2)]{Fri}. 

\'{E}taleness can also be verified directly by proving that for every geometric point $x\in \cK_d$ and $y=p_d(x) \in \cP_d$, the induced homomorphism of complete rings $\widehat{\cO}_{\cK_d,x}\to \widehat{\cO}_{\cP_d,y}$ is an isomorphism. 

As explained in \cite[Proof of 5.8]{MaP}, this blows down to verify that the induced morphism of tangent spaces is an isomorphism. If the point $x$ corresponds to a quasi-polarized K3 surface $(X,L)$, its tangent space corresponds to the deformation space $\ext^1(\Omega_X,L^\vee)$, which is isomorphic to the subspace of primitive classes in $H^{1,1}(X)$. 

Given an isomorphism $\alpha:(H^2(X,\ZZ),L)\simeq (\Lambda,\ell)$, we have an induced identification of $H^{1,1}(X)_{prim}$ with the subspace orthogonal to the linear span of $\{\omega, \overline{\omega}\}$ in $\Lambda_d\otimes\CC$, where $\omega$ is any class in $\alpha_\CC(H^{2,0}(X))$. The latter is exactly the tangent space of $D_d$ at $[\omega]$, which is isomorphic to the tangent space of $\cP_d$ at $y$, as $D_d\to\cP_d$ is \'{e}tale.
\begin{remark}
The period map $p_d$ is not an isomorphism. Indeed, the induced map of automorphism groups is not always surjective: consider a quasi-polarized K3 surface $(X,L)$ whose quasi-polarization \emph{is not} a polarization. Then in $H^{1,1}(X)$ there is an element $\delta$ which is the class of a $(-2)$-curve. The automorphism of $H^2(X,\ZZ)$ given by the reflection with respect to $\delta$ defines then an automorphism of $p_d(X)$ that does not come from an automorphism of $(X,L)$ \cite[(1.2)]{Fri}
\end{remark}
\subsection{}
There is a morphism
\[\varphi_d: \cK_d \longrightarrow \cM_d \]
which sends a quasi-polarized K3 surface $(X\to S,L)$ to the image of $X$ via the map associated to the linear system $|L^{\otimes N}|$ for $N\geq 3$. The image is a polarized K3 surface with at most rational double points. The rational double points arise because of the $(-2)$-curves that get contracted by the polarization.
\subsection{}
Call $\cM_{d}^{\sm}$ the open substack of $\cM_{d}$ corresponding to smooth surfaces. The complement $\cM_{d}^{\sing}$ of $\cM_{d}^{\sm}$, with the reduced scheme structure, is a closed substack and a divisor. Set
\[\cM_d^\sm:=\cM_d\smallsetminus\cM_d^\sing,\quad \cK_d^\sm := \varphi_d^{-1}(\cM_d^\sm).\]
Let $D_d^\sm \subset D_d$ be the open subset of $D_d$ formed by those $[\omega]$ such that the in the sublattice $\omega^{\perp}\cap\overline{\omega}^{\perp} \subset \Lambda_d$ there are no elements $\delta$ such that $\delta^2=(-2)$. This open subset is $\Gamma_d$-invariant, hence we can define
\[\cP_d^\sm := [\Gamma_d\diagdown D_d^{\sm}]. \]
\begin{lemma}\label{lm:factorization}
There exists a factorization
\[
\begin{tikzcd}
\cK_d \ar[r, "p_d"] \ar[rr, bend left, "\varphi_d"] & \cP_d \ar[r, "\psi_d"] & \cM_d
\end{tikzcd}
\]
\end{lemma}
\begin{proof}
We want to show that $\varphi_d$ descends along $p_d$. The latter is \'etale, so all we have to do is to check that there is an isomorphism $\pr_1^*\varphi_d \simeq \pr_2^*\varphi_d$ on $\cK_d\times_{\cP_d}\cK_d$, where the $\pr_i$ denote the two projections.

As $\cM_d$ is separated, if $\pr_1^*\varphi_d$ is isomorphic to $\pr_2^*\varphi_d$ on the generic point, they are isomorphic everywhere. Therefore, in order to conclude is enough to observe that $\cK_d^\sm\to\cM_d^\sm$ and $\cK_d^\sm\to\cP_d^\sm$ are both isomorphisms, so we have that $\varphi_d$ descends to $\varphi_d\circ p_d^{-1}$ along $\cK_d^\sm \to \cP_d^\sm$.

This implies that there is an isomorphism $\pr_1^*\varphi_d\simeq\pr_2^*\varphi_d$ on the generic point.
\end{proof}

Denote by $\cM_{d}^{\ssi}$ the open subset of surfaces with a single $A_{1}$-singularity. Since the deformation theory of a K3 surface with rational double points is unobstructed, and the map from the deformation space of the surface to that of the singularities is smooth, we have that $\cM_{d}^{\sing}$ is a reduced  divisor, and $\cM_{d}^{\ssi}$ is a dense open substack contained in the smooth locus of $\cM_{d}^{\sing}$.

\begin{lemma}\label{lem:irreducible}
The divisor $\cM_{d}^{\sing}$ has two irreducible components if $\frac{d}{2}\equiv 1\pmod 4$, and is irreducible otherwise.
\end{lemma}

\begin{proof}
First we prove a similar statement for $\cP_d^{\sing}$: indeed, if we look at the action of $\Gamma_d$ on the set of generic points of the $\Gamma_d$-invariant divisor $D_d^{\sing} \subset D_d$, we see that when $\frac{d}{2}\equiv 1\pmod 4$ this set is made of two orbits, and is made of one orbit otherwise \cite[Proposition 2.11]{Deb}. This implies that the substack $\cP_d^{\sing}$ has either two or one irreducible components.

To conclude, observe now that $\cP_d^{\sing}$ and $\cM_d^{\sing}$ share the same coarse space, hence they must have the same number of irreducible components.
\end{proof}

%\subsection{}
%In what follows we denote $\cM_d^o$ the open substack of $\cM_d$ formed by polarized K3 surfaces with \emph{at most} one singular point. We also set $\cP_d^o:=\psi_d^{-1}(\cM_d^o)$.
\section{Computation of the Picard group of $\cM_d$} \label{sec:computation 1}
\subsection{}
In this section we compute the Picard group of $\cM_d$, the stack of polarized K3 surfaces of degree $d$ with at most rational double points.
Let $\rho(d)$ be the rank of the rational Picard group of $M_d$ \cite[Corollary 1.3]{BLMM}. Then the main result of this section is the following.
\begin{theorem}\label{thm:pic Md}
We have
    \[ \pic(\cM_d) \simeq \ZZ^{\rho(d)}. \]
\end{theorem}
\subsection{}
By definition $D_d\to\cP_d$ is a $\Gamma_d$-torsor. We can use it to define an analytic line bundle on $\cP_d$ as follows: take $D_d\times \bA^1$ and let $\Gamma_d$ acts diagonally, where the action on $\bA^1$ is given by $A\cdot \lambda:=\det(A)\lambda$. The resulting quotient $\cL_d:=[\Gamma_d\diagdown D_d\times\bA^1]$ is a line bundle over $\cP_d$, which is not trivial because the determinant of an element in $\Gamma_d$ is not trivial in general. Observe also that $\cL_d^{\otimes 2}\simeq\cO_{\cP_d}$, because the determinant of an element in $\Gamma_d$ is a square root of the unity.
\begin{proposition}\label{prop:2 torsion Pic Pd}
The analytic line bundle $\cL_d$ is algebraic and we have
\[ \pic(\cP_d)[n] \simeq \left\{ 
\begin{matrix}
\ZZ/2\ZZ \cdot [\cL_d] && \text{ if }n\text{ is even ,}\\
0 && \text{ if }n\text{ is odd,}
\end{matrix}\right. \]
where $\pic(\cP_d)[n]$ denotes the $\ZZ$-submodule of elements annihilated by $n$.
\end{proposition}
\begin{proof}
Let $\mu_n$ denote the group of $n$-roots of unity. Then in the \'{e}tale topology we have a short exact sequence of sheaves
\[ 0 \longrightarrow \mu_n \longrightarrow \cO^* \overset{(-)^n}{\longrightarrow} \cO^* \longrightarrow 0 \]
where the morphism $\cO^* \to \cO^*$ sends $x$ to $x^n$. By looking at the induced long exact sequence in \'{e}tale cohomology, we have
\begin{equation}\label{eq:seq} 
H^0(\cP_d,\cO^*) \to H^0(\cP_d,\cO^*) \to H^1(\cP_d,\mu_n) \to H^1(\cP_d,\cO^*) \to H^1(\cP_d,\cO^*). 
\end{equation}
Observe that $H^0(\cP_d,\cO^*)=\CC^*$. Indeed, consider the coarse moduli space $\pi:\cP_d\to\cM_d\to M_d$, and its Baily-Borel compactification $\overline{M}_d$: the latter is a normal projective variety \cite[Theorem 10.11]{BB}, and $\overline{M}_d\smallsetminus M_d$ has codimension $>2$. This implies that $\cO(M_d)\simeq\cO(\overline{M}_d)=\CC$; as $M_d$ is a coarse space for $\cP_d$, we have $\pi_*\cO_{\cP_d}\simeq\cO_{M_d}$, from which our claim follows.

In particular, the first arrow in (\ref{eq:seq}) is surjective because $\CC$ is algebraically closed. As $H^1(\cP_d,\cO^*)\simeq \pic(\cP_d)$, we deduce that $\pic(\cP_d)[n]\simeq H^1(\cP_d,\mu_n)$. The latter group classifies cyclic covers of $\cP_d$, which are also classified by surjective homomorphisms $\pi_1(\cP_d)\to \ZZ/n\ZZ$.

As $D_d$ is simply connected, we deduce that $\pi_1(\cP_d)\simeq \Gamma_d$. Any morphism $\Gamma_d\to\ZZ/n\ZZ$ factors through the abelianization of $\Gamma_d$, which is isomorphic to $\ZZ/2\ZZ$ \cite[Theorem 1.7]{GHS}. From this we deduce that $\pic(\cP_d)[n]$ is trivial if $n$ is odd, and isomorphic to $\ZZ/2\ZZ$ if $n$ is even.

Let $\cP_d(\CC)$ denote the anaytic stack associated to $\cP_d$, and let $\cO^*_{an}$ denote the sheaf of invertible holomorphic functions on $\cP_d(\CC)$. Then we have we have a commutative diagram
\[
\begin{tikzcd}
    H^1(\cP_d,\mu_2) \ar[r] \ar[d] & H^1(\cP_d,\cO^*) \ar[d] \\
    H^1_{cl}(\cP_d(\mathbb{C}),\mu_2) \ar[r] & H^1_{cl}(\cP_d(\mathbb{C}),\cO_{an}^*)
\end{tikzcd}
\]
where $H_{cl}(-,-)$ on the bottom row are the sheaf cohomology groups with respect to the analytic topology. By Artin's comparison theorem \cite[Expos\'{e} XI, Theorem 4.4.(iii)]{SGA} the first vertical arrow is an isomorphism: this implies that $\cL_d$ is in the image of $H^1(\cP_d,\mu_2)\to  H^1_{cl}(\cP_d(\mathbb{C}),\cO_{an}^*)$, from which we deduce that it is in the image of the right vertical arrow, i.e. $\cL_d$ comes from an algebraic line bundle, which has to be unique because of our previous computation.
\end{proof}
\begin{lemma}\label{lemma:does not descend}
    The line bundle $\cL_d$ on $\cP_d$ does not descend to a line bundle on $\cM_d$.
\end{lemma}
\begin{proof}
The stabilizer of a generic point of $\cP_d^\sing$ is $\mu_2$, generated by the automorphism given by the reflection $\sigma$ with respect to the unique (up to scalar) element $\delta\in\Lambda_d$ with $\delta^2=(-2)$. This automorphism does not come from an automorphism of the associated singular K3 surface \cite[Remark 1.3]{Fri}. Therefore, if we show that $\sigma$ acts non-trivially on a generic fiber of $\cL_d|_{\cP_d^{\sing}}$, we can conclude that $\cL_d$ does not come from $\cM_d$.

Recall that $\cL_d$ is constructed using the determinant representation of $\Gamma_d$: then it follows that $\sigma_d$ acts via the determinant on a generic fiber of $\cL_d|_{\cP_d^{\sing}}$, and $\det(\sigma)=-1$; we deduce that the action of $\sigma_d$ is not trivial and thus the lemma is proved.
\end{proof}
\begin{lemma}\label{lem:injective}
We have 
\[\varphi_{d*}\cO_{\cK_{d}} = \cO_{\cM_{d}}, \quad p_{d*}\cO_{\cK_{d}} = \cO_{\cP_{d}}.\]
\end{lemma}
\begin{proof}
The first statement follows from the fact that $\psi_{d}$ is proper and birational, and $\cM_{d}$ is smooth.

For the second, the point is that $p_{d}$ is representable, \'{e}tale, surjective and birational. Let $U \arr \cP_{d}$ be an \'{e}tale map, where $U$ is a scheme. Set $V \eqdef U\times_{\cP_{d}}\cK_{d}$; then $V$ is an algebraic space, the map $V \arr U$ is \'{e}tale and a homeomorphism; we need to show that the induced homomorphism $\cO(U) \to \cO(V)$ is an isomorphism. 

If $U^{\sing}$ and $V^{\sing}$ are the inverse images of $\cP^{\sing}$ in $U$ and $V$ respectively, and set $U' \eqdef U \setminus U^{\sing}$ and $V' \eqdef V \setminus V^{\sing}$. Then the restriction $V' \arr U'$ of the projection $V \arr U$ is an isomorphism; hence $\cO(U') \arr \cO(V')$ is an isomorphism. But $\cO(U)$ is the subring of $\cO(U')$ of function without poles on $U^{\sing}$, and analogously $\cO(V$ is the subring of $\cO(V')$ of function without poles on $V^{\sing}$. Since $V \arr U$ is \'{e}tale and surjective, a function in $\cO(U')$ has poles on $U^{\sing}$ if and only if its pullback to $V'$ has poles along $V^\sing$; this completes the proof.
\end{proof}
\begin{proposition}\label{lm:inj}
    The pullback homomorphism $\psi_d^*:\pic(\cM_d)\to\pic(\cP_d)$ is injective. 
\end{proposition}
\begin{proof}
    \Cref{lem:injective} implies that $\cO_{\cM_d}\to\psi_{d*}\cO_{\cP_d}$ is an isomorphism. Therefore, given a line bundle $\cL$ such that $\psi_d^*\cL\simeq\cO_{\cP_d}$, applying the projection formula we have
    \[ \cL\simeq\cL\otimes\cO_{\cM_d}\simeq\cL\otimes\psi_{d*}\cO_{\cP_d}\simeq \psi_{d*}(\psi_d^*\cL\otimes\cO_{\cP_d})\simeq\psi_{d*}\cO_{\cP_d}\simeq\cO_{\cM_d}. \]
\end{proof}
\begin{corollary}\label{cor:torsion free}
    The Picard group of $\cM_d$ is torsion free.
\end{corollary}
\begin{proof}
    By \Cref{lm:inj} the pullback of any non-trivial torsion line bundle on $\cM_d$ is a non-trivial torsion line bundle on $\cP_d$. The only non-trivial torsion line bundle on $\cP_d$ is $\cL_d$, which by \Cref{lemma:does not descend} does not come from $\cM_d$, thus there are no non-trivial torsion line bundles on $\cM_d$.
\end{proof}

\begin{lemma}\label{lm:fg}
    The Picard group of $\cM_d$ is finitely generated.
\end{lemma}
\begin{proof}
    Applying excision to the pair $\cM_d^{\sing}\subset\cM_d$, we see that there is an exact sequence 
    \[ \ZZ^{\oplus e} \longrightarrow \pic(\cM_d) \longrightarrow \pic(\cM_d^{\sm}) \longrightarrow 0 \]
    where $e$ is the number of irreducible components of $\cM_d^{\sing}$, which is either $1$ or $2$ by \Cref{lem:irreducible}. Therefore, is enough to prove that $\pic(\cM_d^{\sm})$ is finitely generated.
    
    We have an isomorphism $\cM_d^{\sm}\simeq[H_d/\PGL_n]$, where $H_d$ is a smooth quasi-projective variety \cite[Example 4.5]{Huy}. It always exists a $\PGL_n$-representation $V$ such that $\PGL_n$ acts freely on an open subset $U\subset V$ whose complement has codimension $\geq 2$ \cite[Lemma 9]{EG}. Consider then $X_d:=[H_d\times U/\PGL_d]$: we claim that (1) $X_d$ is a scheme and (2) $\pic(X_d)\simeq \pic(\cM_d^{\sm})$. 

    Claim (1) can be proved as follows: the quotient stack $[U/\PGL_n]$ is actually a scheme \cite[Lemma 9]{EG}, the group $\PGL_n$ is connected and $H_d$ is smooth, hence normal; from this it follows that $[H_d\times U/\PGL_n]$ is a (smooth) scheme \cite[Proposition 23.(2)]{EG}.

    Claim (2) follows from the fact that $\cM_d^{\sm}\simeq [H_d/\PGL_n]$ is a smooth quotient stack, hence we can identify its Picard group with its equivariant Picard group $\pic^{\PGL_n}(H_d)$ \cite[Proposition 18]{EG}. By homotopy invariance of equivariant Picard groups we have $\pic^{\PGL_n}(H_d)\simeq \pic^{\PGL_n}(H_d\times V)$ \cite[Lemma 2.(b)]{EG}. As the complement of $H_d\times U$ in $H_d\times V$ has codimension $\geq 2$, by excision we deduce $\pic^{\PGL_n}(H_d\times V)\simeq \pic^{\PGL_n}(H_d\times U)$, and the latter is isomorphic to $\pic([H_d\times U/\PGL_n])\simeq \pic(X_d)$.
    
    This implies that $\pic(\cM_d)\simeq\pic(X_d)$, so we reduce to proving the lemma in the case of a smooth quasi-projective variety.

    There exists a smooth compactification $Y_d\supset X_d$ (\cite{Nag} and \cite{Hir}), so if we prove that $\pic(Y_d)$ is finitely generated, we are done.

    For this, observe that $\underline{\pic}_{Y_d}$ is an abelian group scheme over $\CC$, and we claim that $\underline{\pic}_{Y_d}^0$ is an abelian variety with finitely many torsion points: if this is the case, then we can conclude that it is trivial.

    To see that $\underline{\pic}_{Y_d}^0$ has finitely many torsion points, consider the open embedding $X_d\hookrightarrow Y_d$ and the induced pullback homomorphism of groups $\pic^0(Y_d) \to \pic^0(X_d)$: if we prove that the latter has finitely many torsion points, we are done, because the complement of $X_d$ in $Y_d$ is made of finitely many divisors. But we just proved that $\pic(X_d)\simeq\pic(\cM_d^{\sm})$ which is torsion free, hence our claim holds true.

    The claim implies that $\pic(Y_d)$ injects into the Neron-Severi group, which is finitely generated. This concludes the proof.
\end{proof}

\begin{proof}[Proof of \Cref{thm:pic Md}]
    By \Cref{lm:fg} and \Cref{cor:torsion free} we know that $\pic(\cM_d)$ is a finitely generated, torsion free abelian group. Its rank is equal to the rank of $\pic(\cM_d)\otimes\bQ$, and the latter group is isomorphic to $\pic(M_d)\otimes\bQ$, whose rank is known \cite[Corollary 1.3]{BLMM}.
\end{proof}

We can also easily compute the Picard group of the moduli space $M_d$ of (primitively) polarized K3 surfaces of degree $d$ with at most rational double points, leveraging \Cref{cor:torsion free}. 
\begin{corollary}\label{cor:pic Md}We have
$$
\pic(M_d)\cong\mathbb Z^{\rho(d)}.
$$
\end{corollary}

\begin{proof}Since $\rho(d)$ is the rank of $\pic(M_d)\otimes\mathbb Q$, it is enough to show that the Picard group of $M_d$ is torsion free. By \cite[Proposition 6.1]{Ols-integral-models}, the coarse moduli space $\pi:\cM_d\to M_d$ induces an injection of Picard groups $\pi^*:\pic(M_d)\hookrightarrow\pic(\cM_d)$. By Corollary \ref{cor:torsion free}, the latter group is torsion-free and, so, $\pic(M_d)$ is also torsion-free. 
\end{proof}

\section{The Picard groups of $\cP_{d}$ and $\cK_{d}$}\label{sec:computation 2}
\subsection{}
In this last section, we leverage our knowledge of the Picard group of $\cM_d$ to compute the Picard groups of $\cK_d$ and $\cP_d$.

From the fact that $\pic(\cP_{d})$ is finitely generated, which is proved exactly like in the case of $\cM_{d}$ (see Lemma~\ref{lm:fg}), and from Proposition~\ref{prop:2 torsion Pic Pd} we immediately obtain the following.

\begin{theorem}\label{thm:pic Pd abstract}
As an abstract group, $\pic(\cP_{d}) \simeq \ZZ^{\rho(d)} \oplus \ZZ/2$.
\end{theorem}

The exact relation between $\pic(\cM_{d})$, $\pic(\cP_{d})$ and $\pic(\cK_{d})$ depends on whether $\frac{d}{2} \equiv 1 \pmod 4$ or not.

\begin{theorem}\label{thm:pic Pd}
Suppose that $\frac{d}{2} \not\equiv 1 \pmod 4$. Then the pullback $p_{d}^{*}\colon \pic(\cP_{d}) \arr \pic(\cK_{d})$ is an isomorphism and we have a split exact sequence
   \[
   \begin{tikzcd}
   0 \rar &{\pic\cM_{d}} \rar["\varphi_{d}^{*}"] & {\pic\cP_{d}} \rar & \ZZ/2\rar & 0\,.
   \end{tikzcd}
   \]
where the splitting is given by $1\mapsto [\cL_d]$.
\end{theorem}

\begin{theorem}\label{thm:pic Pd new}
    Suppose $\frac{d}{2} \equiv 1 \pmod 4$. Then $p_d^*\colon\pic(\cP_d)\to\pic(\cK_d)$ is an isomorphism; furthermore we have a non-split short exact sequence
    \[
   \begin{tikzcd}
   0 \rar &{\pic\cM_{d} \times \ZZ/2} \rar["\varphi_{d}^{*}\times \id"] & {\pic\cP_{d}} \rar & \ZZ/2\rar & 0\,.
   \end{tikzcd}
   \]
    and neither class $[\cP_{d,1}^{\sing}]$ or $[\cP_{d,2}^{\sing}]$ sent to zero by the last map. 
\end{theorem}

\subsection{}
Call $\cP_{d}^{\sing}$ and $\cK_{d}^{\sing}$ the inverse images of $\cM_{d}^{\sing}$ in $\cP_{d}$ and $\cK_{d}$ respectively, with their reduced scheme structure. Since $\cK_{d} \arr \cP_{d}$ is \'{e}tale, $\cK_{d}^{\sing}$ is the scheme-theoretic inverse image of $\cP_{d}^{\sing}$. Moreover, when $\frac{d}{2}\equiv 1\pmod 4$, let $\cM_{d,i}^{\sing}$ (resp. $\cP_{d,i}^{\sing}$, $\cK_{d,i}^{\sing}$) for $=1,2$ be the two irreducible components of $\cM_d^{\sing}$ (resp. $\cP_{d}^{\sing}$, $\cK_{d}^{\sing}$) given by \Cref{lem:irreducible}. 

\begin{lemma}\label{lem:ramification}
We have
   \[\psi_{d}^{*}[\cM_{d}^{\sing}] = 2[\cP_{d}^{\sing}], \quad  p_{d}^{*}[\cP_{d}^{\sing}] = [\cK_{d}^{\sing}].  \]
If $\frac{d}{2}\equiv 1\pmod 4$ we have
   \[\psi_{d}^{*}[\cM_{d,i}^{\sing}] = 2[\cP_{d,i}^{\sing}], \quad  p_{d}^{*}[\cP_{d,i}^{\sing}] = [\cK_{d,i}^{\sing}].   \]
\end{lemma}

\begin{proof}
The second equations of both statements follow from the fact that $p_{d}$ is \'{e}tale.

For the first, notice that a generic closed point $\xi\colon \spec \CC \arr \cM_{d}^{\sing}$ corresponds to a polarized K3 surface $(X, L)$ with only one singular point of type $A_{1}$. If $\widetilde{X} \arr X$ is the crepant resolution of $X$, we can fix an isomorphism $\alpha:H^{2}(\widetilde{X}, \ZZ) \simeq \Lambda$ sending $\rmc_{1}(L)$ in $\ell$; by taking $\alpha\otimes\id_\CC(H^{2,0}(X))$, we obtain a point $\gamma\in D_d$ mapping to $\xi$, hence a lifting $\eta\colon \spec \CC \arr \cP_{d}$ of $\xi$. 

The automorphism group of $\eta$ is the stabilizer of $\gamma$ in $\Lambda_{d}$, which contains the reflexion along the element $\delta\in \Lambda$ corresponding to the class in $H^{2}(X, \ZZ)$ of the $(-2)$-curve contracted by $\widetilde{X} \arr X$. Then the group homomorphism $\aut \eta \arr \aut\xi$ is surjective, and its kernel is cyclic of order $2$, generated by the reflexion along $\delta$. 

It follows that $\psi_{d}$ is ramified of order~$2$ along the irreducible components $\cP_{d}^{\sing}$, which implies the first equations of both statements.
\end{proof}

\begin{lemma}\label{lem:image}
For $=1,2$, the divisor $[\cP_{d,i}^{\sing}]$ does not belong to the image of $\psi_d^*:\pic(\cM_d)\to\pic(\cP_d)$.
\end{lemma}
\begin{proof}
Let us assume that $i=1$, the other case can be proved in the same way. We argue by contradiction, thus suppose that $[\cP_{d,1}^{\sing}]=\psi_d^*[\cL']$. This implies that the ideal sheaf $\cO(-\cP_{d,1}^{\sing})$ comes from $\cM_d$.

Let $x\in \cP_{d,1}^{\sing}$ be a generic point, and let $\sigma_d\in \aut(x)$ be the involution that does not come from $\aut(\psi_d(x))$.
Then $\cO(-\cP_{d,1}^{\sing})(x)$ is generated by a local equation of $\cP_{d,1}^{\sing}$ on which $\sigma_d$ should act trivially. We now show that this is not the case.

Let $y\in D_d$ be a point mapping to $x$ in $\cP_d$: then there exists a $(-2)$-class $\delta \in \Lambda_d$ such that 
\[\ell(z)=\langle \delta, z \rangle = 0\]
is a local equation for the preimage of $\cP_{d,1}^{\sing}$ around $y$. By construction, the involution $\sigma_d$ corresponds to the reflection with respect to the hyperplane $\delta^{\bot}$, hence it maps $\delta\mapsto -\delta$. This implies that $\sigma \cdot \ell(z) = -\ell(z)$, hence the action on the generator of $\cO(-\cP_{d,1}^{\sing})$ is not trivial. We have reached a contradiction.
\end{proof}

\begin{proof}[Proof of Theorem~\ref{thm:pic Pd}]
The fact that $\psi_{d}^{*}\colon \pic(\cM_{d}) \arr \pic(\cP_{d})$ and $p_{d}^{*}\colon \pic(\cP_{d}) \arr \pic(\cK_d) $ are injective follows from Lemma~\ref{lem:injective} and the projection formula.

By \Cref{lem:irreducible} we have that $\cM_d^{\sing}$, $\cP_d^{\sing}$, $\cK_d^{\rm sing}$ are irreducible. We have three homomorphisms $\ZZ \arr \pic(\cM_{d})$, $\ZZ \arr \pic(\cP_{d})$ and $\ZZ \arr \pic(\cK_{d})$ sending $1 \in \ZZ$ in $\cM_{d}^{\sing}$, $\cP_{d}^{\sing}$ and $\cK_{d}^{\sing}$ respectively. From Lemma~\ref{lem:ramification} we get two commutative diagrams with exact rows
   \[
   \begin{tikzcd}
   \ZZ \rar\dar["\cdot 2"] & {\pic(\cM_{d})} \dar["\psi_{d}^{*}"]\rar & {\pic(\cM_{d}^{\sm}})\rar\dar[equal] & 0\\
   \ZZ \rar & {\pic(\cP_{d})} \rar & {\pic(\cM_{d}^{\sm})}\rar & 0
   \end{tikzcd}
   \]
and
   \[
   \begin{tikzcd}
   \ZZ \rar\dar[equal] & {\pic(\cP_{d})} \dar["p_{d}^{*}"]\rar & {\pic(\cP_{d}^{\sm})}\rar\dar[equal] & 0\\
   \ZZ \rar & {\pic(\cK_{d})} \rar & {\pic(\cK_{d}^{\sm})}\rar & 0
   \end{tikzcd}
   \]
From the second we get that $\psi_{d}^{*}\colon \pic(\cP_{d}) \arr \pic(\cK_{d})$ is surjective, hence, that is is an isomorphism, as claimed.

From the first we obtain that the cokernel of the injective map $\psi_{d}^{*}\pic(\cM_{d}) \arr \pic(\cP_{d})$ is contained in the cokernel of $\ZZ\xarr{\cdot 2} \ZZ$, which is $\ZZ/2\ZZ$. Since we have a $2$-torsion element of $\pic(\cP_{d})$, the class of $\cL_{d}$, that does not come from $\pic(\cM_{d})$ (Lemma~\ref{lemma:does not descend}), this proves the result.
\end{proof}

\begin{proof}[Proof of \Cref{thm:pic Pd new}]
The fact that $\psi_{d}^{*}\colon \pic(\cM_{d}) \arr \pic(\cP_{d})$ and $p_{d}^{*}\colon \pic(\cP_{d}) \arr \pic(\cK_d) $ are injective follows from Lemma~\ref{lem:injective} and the projection formula, as in the previous proof.

We have $\cM_d^{\sing}=\cM_{d,1}^{\sing}\cup \cM_{d,2}^{\rm sing}$, where the $\cM_{d,i}^{\rm sing}$ are integral divisors. Similarly, we have $\cP_d^{\sing}=\cP_{d,1}^{\sing}\cup \cP_{d,2}^{\rm sing}$ and $\cK_d^{\sing}=\cK_{d,1}^{\sing}\cup \cK_{d,2}^{\rm sing}$.

Define a homomorphism $\ZZ\cdot e_1 \oplus \ZZ\cdot e_2 \to \pic(\cM_d)$ given by $e_i \mapsto [\cM_{d,i}^{\sing}]$. We also have homomorphisms $\ZZ\cdot e_1 \oplus \ZZ\cdot e_2\to \pic(\cP_d)$ and $\ZZ\cdot e_1 \oplus \ZZ\cdot e_2\to \pic(\cK_d)$ defined in a similar way. From the second part of \Cref{lem:ramification} we have commutative diagrams
 \[
   \begin{tikzcd}
   \ZZ\cdot e_1\oplus \ZZ\cdot e_2 \rar\dar["\cdot 2"] & {\pic(\cM_{d})} \dar["\psi_{d}^{*}"]\rar & {\pic(\cM_{d}^{\sm}})\rar\dar[equal] & 0\\
   \ZZ\cdot e_1\oplus \ZZ\cdot e_2 \rar & {\pic(\cP_{d})} \rar & {\pic(\cM_{d}^{\sm})}\rar & 0
   \end{tikzcd}
   \]
and
   \[
   \begin{tikzcd}
   \ZZ\cdot e_1\oplus \ZZ\cdot e_2 \rar\dar[equal] & {\pic(\cP_{d})} \dar["p_{d}^{*}"]\rar & {\pic(\cP_{d}^{\sm})}\rar\dar[equal] & 0\\
   \ZZ\cdot e_1\oplus \ZZ\cdot e_2 \rar & {\pic(\cK_{d})} \rar & {\pic(\cK_{d}^{\sm})}\rar & 0.
   \end{tikzcd}
   \]
From the second we get that $\psi_{d}^{*}\colon \pic(\cP_{d}) \arr \pic(\cK_{d})$ is surjective, hence, that is is an isomorphism, as claimed.

From the first, we get that there is an exact sequence
\[ 0 \longrightarrow \pic(\cM_d) \longrightarrow \pic(\cP_d) \overset{f}{\longrightarrow} \ZZ/2\times\ZZ/2. \]
Observe that $f(\cL_d)\neq 0$ because $\cL_d$ does not come from $\pic(\cM_d)$. From this we deduce that we have an exact sequence
\[ 0 \longrightarrow \pic(\cM_d)\times\ZZ/2\cdot[\cL_d] \longrightarrow \pic(\cP_d) \overset{g}{\longrightarrow} \ZZ/2. \]
To prove that the last arrow is surjective, we show that $g([\cP_{d,1}^{\sing}])\neq 0$ (the same argument applies also to $[\cP_{d,2}^{\sing}]$).

We argue by contradiction: if $g([\cP_{d,1}^{\sing}])= 0$, then
\[[\cP_{d,1}^{\sing}] = \psi_d^*[\cL'] + a[\cL_d], \quad a\in\{0,1\}. \]
Suppose first $a=1$. If we restrict everything to the open substack $\cP_d \smallsetminus \cP_{d,1}^{\sing}$, we deduce that $\cL_d|_{\cP_d \smallsetminus \cP_{d,1}^{\sing}}$ comes from $\cM_d \smallsetminus\cM_{d,1}^{\sing}$; this is a contradiction, because by construction the automorphism group of a point $x\in\cP_{d,2}^{\sing}$ acts non trivially on the fiber $\cL_d(x)$.

This shows that we must have $a=0$, but also this cannot be the case because of \Cref{lem:image}. Therefore $g([\cP_{d,1}^{\sing}])\neq 0$ as claimed. The fact that $g$ is not split follows from the fact that $\pic(\cP_d)[2]\simeq \ZZ/2$ by \Cref{prop:2 torsion Pic Pd}.
\end{proof}

\printbibliography
\end{document}